\documentclass[11pt]{amsart} 
\usepackage{amsmath,amssymb,latexsym,esint,cite,mathrsfs}
\usepackage{verbatim,wasysym}
\usepackage[left=2.85cm,right=2.85cm,top=2.5cm,bottom=2.5cm]{geometry}

\usepackage{microtype}
\usepackage{color,enumitem,graphicx}
\usepackage[colorlinks=true,urlcolor=blue, citecolor=red,linkcolor=blue,
linktocpage,pdfpagelabels, bookmarksnumbered,bookmarksopen]{hyperref}
\usepackage[english]{babel}
\usepackage{lipsum}
\usepackage{braket}

\newtheorem{lemma}{Lemma}[section]

\newtheorem{proposition}[lemma]{Proposition}
\newtheorem{theorem}[lemma]{Theorem}

\theoremstyle{definition}

\numberwithin{equation}{section}

\newcommand{\R}{\mathbb{R}}

\newcommand{\esssup}{\operatornamewithlimits{ess\,sup}}
\newcommand{\essinf}{\operatornamewithlimits{ess\,inf}}

\renewcommand{\Gamma}{\varGamma}

\newcommand{\supt}{\operatorname{supt}}
\newcommand{\Sf}{{\mathbb S^{n-1}}}

\def\XXint#1#2#3{{\setbox0=\hbox{$#1{#2#3}{\int}$} 
		\vcenter{\hbox{$#2#3$}}\kern-.5\wd0}}

\title[Approximations of anisotropic Sobolev norms]{Nonlocal approximations to anisotropic Sobolev norms}

\author[I.\ Cinelli]{Ivan Cinelli}
\author[G.\ Ferrari]{Gianluca Ferrari}
\author[M.\ Squassina]{Marco Squassina}

\address[I.\ Cinelli]{Dipartimento di Matematica e Fisica
	\newline\indent
	Universit\`a Cattolica del Sacro Cuore
	\newline\indent
	Via della Garzetta 48, Brescia, Italy}
\email{ivan.cinelli01@icatt.it}

\address[G.\ Ferrari]{Dipartimento di Matematica e Fisica
	\newline\indent
	Universit\`a Cattolica del Sacro Cuore
	\newline\indent
	Via della Garzetta 48, Brescia, Italy}
\email{gianluca.ferrari03@icatt.it}

\address[M.\ Squassina]{Dipartimento di Matematica e Fisica
	\newline\indent
	Universit\`a Cattolica del Sacro Cuore
	\newline\indent
	Via della Garzetta 48, Brescia, Italy}
\email{marco.squassina@unicatt.it}

\thanks{The third author is member
	of {\em Gruppo Nazionale per l'Analisi Ma\-te\-ma\-ti\-ca, la Probabilit\`a e le loro Applicazioni} (GNAMPA) 
	of the {\em Istituto Nazionale di Alta Matematica} (INdAM)}

\begin{document}

\begin{abstract}
We obtain some nonlocal characterizations for a class of variable exponent Sobolev spaces 
arising in nonlinear elasticity, in the theory of electrorheological fluids as well as in image 
processing for the regions where the variable exponent $p(x)$ reaches the value $1$. 
\end{abstract}

\maketitle

\section{Introduction}

Let $\Omega$ be a smooth bounded domain, $p\geq 1$ and $\varphi : [0,+\infty) \to [0,+\infty)$ be a continuous function, except at a finite number of points, such that $\varphi(0) = 0$ and 
\begin{equation*}
	\gamma_{n,p}\int_0^{+\infty} \varphi(t)\,t^{-(p+1)}\,dt=1,
\end{equation*}
where 
$$
\gamma_{n,p}=\int_{\Sf}|\omega\cdot\mathbf{e}|^{p}\,d\mathcal{H}^{n-1}(\omega) , \quad \mathbf{e}\in\Sf.
$$
If $\varphi$ satisfies further growth conditions, Brezis and Nguyen proved recently in \cite{BNg20} that for 
every function $u\in W^{1,p}(\Omega)$ with $p>1$
$$
\lim_{\delta\to0^+} \delta^{p}\int_\Omega \int_\Omega \frac{\varphi \left(|u(x)-u(y)|/\delta\right)}{|x-y|^{n+p}}\,dx\,dy =\int_{\Omega}|\nabla u(x)|^{p}\,dx.
$$
Hence the $p$-Dirichlet energy functional 
$$
W^{1,p}(\Omega)\ni u\mapsto \int_{\Omega}|\nabla u(x)|^{p}\,dx
$$
can be approximated by a suitable class of  nonlocal nonconvex energy functionals. 

A similar pointwise convergence result was investigated in \cite{BNg18} for smooth functions  in the 
case $p=1$, relevant for image processing. In this case, the above nonlocal nonconvex functionals are related to image processing theories. For nonsmooth
functions the situation is much more involved and the convergence only holds in the sense of $\Gamma$-convergence
\cite{BNg18}. As a model function for the above results one can think of
\begin{equation}
	\label{modellovarphi}
\varphi(t) = a
\begin{cases}
	t^{p+1}, & \ \mathrm{for} \ 0 \le t \le 1, \\ 
	1, & \ \mathrm{for} \ t>1,
\end{cases}
\qquad a>0.
\end{equation}
Other types of nonlocal nonconvex approximations of  $p$-Dirichlet energies where previously investigated by Nguyen in \cite{nguyen06,NgSob2,Ng11} after more classical convex approximations like the one by Bourgain, Brezis and Mironescu where studied, see \cite{BBNg1,BBM1}. 

On the other hand, differential equations and variational problems involving
variable $p(x)$-growth conditions, and hence variable exponent Sobolev spaces $W^{1,p(\cdot)}$, arise
from nonlinear elasticity theory and electrorheological fluids, and have been the target of various
investigations, especially in regularity theory, see \cite{DHHR,AcMi1,AcMi2}.
A model investigated by Chen, Levine e Rao \cite{rao} was elaborated with the idea 
of  merging parts of the domain where $p(x)=2$ (isotropic diffusion) is a suitable 
choice and parts where $p(x)=1$ (total variation case) is instead more suitable.
They investigated the minimization problem
$$ 
\min \int_{\Omega} | Du|^{p(x)} + \frac{\lambda}{2} \int_{\Omega}|u-f|^{2}dx
$$
with $1 \le p(x) \le 2$ and $\lambda \ge 0$.
It is thus natural to wonder if also the $p(\cdot)$-Dirichlet energy  
$$
W^{1,p(x)}(\Omega)\ni u\mapsto \int_{\Omega}|\nabla u(x)|^{p(x)}\,dx
$$
can be approximated by a suitable class of  nonlocal nonconvex energy functionals.
%of the form
%$$
%u\mapsto \int_\Omega \int_\Omega \delta^{p(x)}\frac{\varphi \left(x,|u(x)-u(y)|/\delta\right)}{|x-y|^{n+p(x)}}\,dx\,dy.
%$$
We will prove that, in fact, under suitable assumptions on the function $\varphi(x,t)$
$$
 \lim_{\delta\to0^+}\int_\Omega \int_\Omega \delta^{p(x)}\frac{\varphi \left(x,|u(x)-u(y)|/\delta\right)}{|x-y|^{n+p(x)}}\,dx\,dy=\int_{\Omega}|\nabla u(x)|^{p(x)}\,dx,
$$
for every $u\in W^{1,p^{+}}\left(\Omega\right)\cap W^{1,p^{-}}\left(\Omega\right)$ in the case $\inf p>1$ and for $u\in C^1(\overline{\Omega})$ in the case $\inf p=1$.
As an example of function $\varphi$ consistent with the meaningful constant case \eqref{modellovarphi},  
one can consider 
$$
\varphi(x,t) = 
\begin{cases}
	a(x)t^{p(x)+1}, & \ \mathrm{for} \ 0 \le t \le 1, \\ 
	b(x), & \ \mathrm{for} \ t>1,
\end{cases}
$$
where $a(x)$ is a measurable function and
$$
b(x) =  p(x)\left(\frac{1}{\gamma_{n,p(x)}} - a(x)\right),\qquad 
a(x) \le \frac{p(x)}{\gamma_{n,p(x)}\left(p(x)+1\right)} \le b(x),
$$
for a.e. $x\in\Omega$,
in order to fulfill the normalization  
\begin{equation}
	\label{Hp4}
\gamma_{n,p(x)}\int_0^{+\infty} \varphi(x,t)\,t^{-(p(x)+1)}\,dt=1, 
\end{equation}
where we have set,  for a.e. $x\in\R^n$,
$$
\gamma_{n,p(x)}:=\int_{\Sf}|\omega\cdot\mathbf{e}|^{p(x)}\,d\mathcal{H}^{n-1}(\omega), \quad \mathbf{e}\in\Sf.
$$ 
Previously, in the recent work \cite{FS-2020}, two of the authors proved a different approximation result 
in the variable exponent case. More precisely, for
$u\in W^{1,p^{+}}\left(\R^n\right)\cap W^{1,p^{-}}\left(\R^n\right)$, then
		$$
		\lim_{\delta\to0} \underset{|u(x)-u(y)|>\delta}{\int_{\R^n}\int_{\R^n}} \frac{\delta^{p(x)}}{|x-y|^{n+p(x)}} \, dx \, dy = \int_{\R^n} K_{n,p(x)} |\nabla u(x)|^{p(x)} \, dx,
		$$
		where we have set
		\begin{equation*}
			K_{n,p(x)} := \frac{1}{p(x)} \int_{\mathbb{S}^{n-1}} \left|\omega\cdot
			\mathbf{e}\right|^{p(x)}\,d\mathcal{H}^{n-1}(\omega) ,\quad \mathbf{e}\in\mathbb{S}^{n-1}.
		\end{equation*}
		This extends the results in \cite{nguyen06}.
		
		 The limitations to
		$u\in W^{1,p^{+}}\left(\Omega\right)\cap W^{1,p^{-}}\left(\Omega\right)$
		in place of $u\in W^{1,p(\cdot)}\left(\Omega\right)$ is related to failure of  the
		continuity inequality for the maximal 
		function when $p(x)$ is not constant. More precisely, taken any $\omega \in \Sf$, we define
		$
		\mathcal{M}_\omega(u)(x) = \sup_{h>0} \frac{1}{h} \int_0^h |u(x+s\omega)|\,ds \,
		$
		as the maximal function of $u$ along the direction $\omega$. By \cite[Lemma 3.1]{NPSV}
		there exists $C>0$ such that, for all $\omega \in \Sf$,
		$$
		\int_{\R^n} |\mathcal{M}_\omega(u)(x)|^p \,dx \le C \int_{\R^n} |u(x)|^p \,dx , \qquad \forall u \in L^p\left(\R^n\right).
		$$ 
		%
		%{\color{red}{
				%\begin{proposition}
					%	Let $p\in\mathcal{P}^{\log}\left(\R^n\right)$ be a bounded variable exponent. Then there exists a
					%	positive constant $C_1$ depending on $n$ and $p$ and a function $\Upsilon:L^{p(\cdot)}\left(\R^n\right)\to\R$ such that, for all $\omega \in \Sf$,
					%	$$
					%	\int_{\R^n} |\mathcal{M}_\omega(u)(x)|^{p(x)} \,dx \le C_1(n,p) \int_{\R^n} |u(x)|^{p(x)} \,dx+\Upsilon(u) , \qquad \forall u \in L^{p(\cdot)}\left(\R^n\right).
					%	$$ 
					%	\end{proposition}
				%}}
		For variable exponents the inequality fails in general \cite{larsp,Izuki}. For instance, if
		$p(x)=2$ on $(-\infty,-2)$ and $p(x)\geq 4$ on $[2,+\infty)$, then $\int_\R |u|^{p(x)}dx<+\infty$, 
		but  $\int_\R |\mathcal{M}(u)|^{p(x)}dx=+\infty$ for the function $u(x)=|x|^{-1/3}\chi_{[2,\infty)}(x)$.

\section{Main results}
 Let us now formulate the main results.
Consider $p:\Omega\to[1,+\infty)$ measurable and set
$$
p^- := \essinf_{\Omega} p  \qquad \text{and} \qquad p^+ := \esssup_{\Omega} p .
$$
We set
$$
W^{1,p^{\pm}}\left(\Omega\right):=W^{1,p^{+}}\left(\Omega\right)\cap W^{1,p^{-}}\left(\Omega\right) ,
$$
Let $\varphi : \R^n\times [0,+\infty) \to [0,+\infty)$ be a continuous function -- except at a finite number of points in $(0,+\infty)$ -- such that $\varphi(x,0) = 0$ for a.e. $x\in\R^n$ and satisfying the following assumptions
\begin{equation}
	\label{Hp1}
	\exists a \ge 0 : \quad \varphi(x,t) \le at^{p(x)+1} ,\quad \forall t \in [0,1],
\end{equation}
\begin{equation}
	\label{Hp2}
	\exists b \ge 0 : \quad \varphi(x,t) \le b ,\quad \forall t \in \R^+ ,
\end{equation}
for a.e. $x\in\R^n$.
In the spirit of \cite{BNg20}, we introduce the following nonlocal functionals
$$
\Lambda_{\delta}(u) := \int_\Omega \int_\Omega \frac{\varphi_{\delta} \left(x,|u(x)-u(y)|\right)}{|x-y|^{n+p(x)}}\,dx\,dy,
$$
where 
$$
\varphi_{\delta}(x,u):=\delta^{p(x)} \varphi(x,u/\delta),\qquad \delta>0.
$$
As anticipated, we will prove that, in the above framework, if $1\leq p^-\le p^+<+\infty$, then 
$$
\lim_{\delta\to0^+} \Lambda_{\delta}(u) =\int_{\Omega}|\nabla u(x)|^{p(x)}\,dx
$$
for $u\in W^{1,p^{\pm}}\left(\Omega\right)$ when $p^->1$
and for $u\in C^1\left(\bar{\Omega}\right)$ when $p^-=1$. As already pointed out this
is not a technical limitation, but a true difference between $p^->1$ and $p^-=1$.
Further sufficient conditions for a function to belong to $W^{1,p(\cdot)}(\Omega)$
can be found in Proposition \ref{suffcond1} and Theorem \ref{prop11}.
The integrals of \cite{nguyen06,NgSob2,Ng11}  corresponds  to $\varphi(x,t)=\chi_{(1,+\infty)}(t)$,
up to small changes in the formulation of the results.

\subsection{The case $p^- > 1$}
We have the following main result
\begin{theorem}
	Let $1<p^-\le p^+<+\infty$ and $u\in W^{1,p^{\pm}}\left(\Omega\right)$. Assume \eqref{Hp4}, \eqref{Hp1}, \eqref{Hp2}. Then
	\begin{enumerate}
		\item[(a)] there exists $C>0$, depending on $p^\pm$ and the domain $\Omega$, such that for every $\delta>0$
		$$
			\Lambda_{\delta}(u) \le C\left(\|\nabla u\|_{L^{p^+}\left(\Omega\right)}^{p^+} + \|\nabla u\|_{L^{p^-}\left(\Omega\right)}^{p^-}\right);
		$$
		\item[(b)] we have
		$$
			\lim_{\delta\to0^+} \Lambda_{\delta}(u) =\int_{\Omega}|\nabla u(x)|^{p(x)}\,dx.
		$$
	\end{enumerate}
%	Moreover, if we assume that $\varphi$ satisfies \eqref{Hp3}, every function $u\in L^{p(\cdot)}\left(\Omega\right)$ such that
%	$$
%	\liminf_{\delta\to0^+} \Lambda_{\delta}(u) < +\infty
%	$$
%	belongs to $W^{1,p(\cdot)}\left(\Omega\right)$.
\end{theorem}
\begin{proof}
	First of all, we prove the results in the case $\Omega = \R^n$.
\item[(a)] By making the change of variables $z=x-y$ and using polar coordinates for $z$, we can write $\Lambda_{\delta}(u)$ as
$$
\begin{aligned}
\Lambda_{\delta}(u) & = \int_{\R^n}\int_{\R^n} \delta^{p(x)} \frac{\varphi\left(x,|u(x)-u(y)|/\delta\right)}{|x-y|^{n+p(x)}}\,dx\,dy \\
& = \int_{\R^n}\int_{\Sf}\int_0^{+\infty} \delta^{p(x)} \frac{\varphi\left(x,|u(x+h\omega)-u(x)|/\delta\right)}{|h\omega|^{n+p(x)}}h^{n-1}\,dh\,d\mathcal{H}^{n-1}(\omega)\,dx \\
& = \int_{\R^n}\int_{\Sf}\int_0^{+\infty} \delta^{p(x)} \frac{\varphi\left(x,|u(x+h\omega)-u(x)|/\delta\right)}{h^{p(x)+1}}\,dh\,d\mathcal{H}^{n-1}(\omega)\,dx.
\end{aligned}
$$
By setting $h=\delta \tilde{h}$ and relabeling $\tilde{h}$ as $h$, we have
$$
\begin{aligned}
	\Lambda_{\delta}(u) & = \int_{\R^n}\int_{\Sf}\int_0^{+\infty} \delta^{p(x)} \frac{\varphi\left(x,|u(x+h\omega)-u(x)|/\delta\right)}{h^{p(x)+1}}\,dh\,d\mathcal{H}^{n-1}(\omega)\,dx \\
	& = \int_{\R^n}\int_{\Sf}\int_0^{+\infty} \delta^{p(x)} \frac{\varphi\left(x,|u(x+\delta \tilde{h}\omega)-u(x)|/\delta\right)}{\left(\delta\tilde{h}\right)^{p(x)+1}}\,\delta d\tilde{h}\,d\mathcal{H}^{n-1}(\omega)\,dx \\
	& = \int_{\R^n}\int_{\Sf}\int_0^{+\infty} \frac{\varphi\left(x,|u(x+\delta h\omega)-u(x)|/\delta\right)}{h^{p(x)+1}}\, dh\,d\mathcal{H}^{n-1}(\omega)\,dx ,
\end{aligned}
$$
so that
$$
\begin{aligned}
	\Lambda_{\delta}(u) := \int_{\R^n} \int_{\R^n} & \frac{\varphi_{\delta} \left(x,|u(x)-u(y)|\right)}{|x-y|^{n+p(x)}}\,dx\,dy \\
	& = \int_{\R^n}\int_{\Sf}\int_0^{+\infty} \frac{\varphi\left(x,|u(x+\delta h\omega)-u(x)|/\delta\right)}{h^{p(x)+1}}\, dh\,d\mathcal{H}^{n-1}(\omega)\,dx .
\end{aligned}
$$
Now, from the fundamental theorem of calculus, the inequalities chain
$$
\begin{aligned}
\frac{\left|u(x+\delta h \omega)-u(x)\right|}{\delta} & \le \frac{1}{\delta} \int_0^{\delta h} \left|\frac{d}{ds} u(x+s\omega)\right|\,ds \\
& = \frac{1}{\delta} \int_0^{\delta h} \left|{\nabla u(x+s\omega)\cdot\omega}\right|\,ds \le h\mathcal{M}_\omega \left(\nabla u\right)(x)
\end{aligned}
$$
follows for a.e. $(x,h,\omega)\in\R^n\times \R^+ \times\Sf$, where $\mathcal{M}_\omega \left(\nabla u\right)(x)$ denotes the maximal function of $\nabla u$ along the direction $\omega \in \Sf$ that, for a general function $f :\R^n\to\R$, is given by
$$
\mathcal{M}_\omega \left(f\right)(x) := \sup_{h>0} \frac{1}{h} \int_0^h \left|f(x+s\omega)\right|\,ds .
$$
Let's set
$$
\tilde{\varphi}(x,t) :=
\begin{cases}
	at^{p(x)+1}\,,		& t\in[0,1) \\
	b \,,					  &  t\in[1,+\infty)
\end{cases}.
$$
$\tilde{\varphi}(x,\cdot)$ is a non-decreasing function for a.e. $x\in\R^n$ and, according to hypothesis \eqref{Hp1} and \eqref{Hp2}, we have
$$
\varphi(x,t)\le\tilde{\varphi}(x,t), \quad \forall t \in \R , \,\,\text{for a.e.}\,\, x\in\R^n.
$$
For these reasons,
$$
\varphi\left(x,\left|u(x+\delta h \omega)-u(x)\right|/\delta\right)\le
\tilde{\varphi}\left(x,\left|u(x+\delta h \omega)-u(x)\right|/\delta\right) \le
\tilde{\varphi}\left(x,h\mathcal{M}_\omega \left(\nabla u\right)(x)\right)
$$
for a.e. $(x,h,\omega)\in\R^n\times \R^+ \times\Sf$, so -- being $\tilde{\varphi}$ a non-negative function -- we can increase $\Lambda_\delta (u)$ as follows
\begin{equation}
\label{eq2}
\begin{aligned}
	\Lambda_\delta (u) & = \int_{\R^n}\int_{\Sf} \int_0^{+\infty} \frac{\varphi\left(x,|u(x+\delta h\omega)-u(x)|/\delta\right)}{h^{p(x)+1}}\, dh\,d\mathcal{H}^{n-1}(\omega)\,dx \\
	& \le \int_{\R^n}\int_{\Sf}\int_0^{+\infty} \frac{\tilde{\varphi}\left(x,h\mathcal{M}_\omega \left(\nabla u\right)(x)\right)}{h^{p(x)+1}}\, dh\,d\mathcal{H}^{n-1}(\omega)\,dx \\
	& = \int_{\R^n} \int_\Sf \int_0^{+\infty} \tilde{\varphi}(x,t) \left(\frac{\mathcal{M}_\omega \left(\nabla u\right)(x)}{t}\right)^{p(x)+1} \frac{1}{\mathcal{M}_\omega \left(\nabla u\right)(x)} \,dt \,d\mathcal{H}^{n-1}(\omega) \,dx \\
	& = \int_{\R^n} \int_0^{+\infty} \tilde{\varphi}(x,t) t^{-(p(x)+1)} \,dt \int_\Sf  \left|\mathcal{M}_\omega \left(\nabla u\right)(x)\right|^{p(x)} \,d\mathcal{H}^{n-1}(\omega) \,dx ,
\end{aligned}
\end{equation}
where we have set $t=h\mathcal{M}_\omega \left(\nabla u\right)(x)$. We are now ready to prove the existence of a constant $C>0$, depending on $p^\pm$, such that
$$
	\Lambda_{\delta}(u) \le C\left(\|\nabla u\|_{L^{p^+}\left(\R^n\right)}^{p^+} + \|\nabla u\|_{L^{p^-}\left(\R^n\right)}^{p^-}\right).
$$
First of all, for how the function $\tilde{\varphi}$ was built, let's observe that, for a.e. $x\in\R^n$, the term
$$
\int_0^{+\infty} \tilde{\varphi}(x,t) t^{-(p(x)+1)} \,dt
$$
is bounded. In fact, recalling that $1 < p^- \le p(x) \le p^+ <+\infty$, we have
$$
\begin{aligned}
	\int_0^{+\infty} \tilde{\varphi}(x,t) t^{-(p(x)+1)} \,dt & = \int_0^1 \left(at^{p(x)+1}\right) t^{-(p(x)+1)} \,dt + \int_1^{+\infty} b t^{-(p(x)+1)} \,dt \\
	& \le a + b \int_1^{+\infty} \frac{1}{t^{p^-+1}} \,dt = \alpha <+\infty ,
\end{aligned}
$$
for a.e. $x\in\R^n$, so that equation \eqref{eq2} becomes
\begin{equation}
	\label{eq3}
	\Lambda_{\delta}(u)\le \alpha \int_{\R^n} \int_\Sf  \left|\mathcal{M}_\omega \left(\nabla u\right)(x)\right|^{p(x)} \,d\mathcal{H}^{n-1}(\omega) \,dx .
\end{equation}
At this point, the integral
$$
\int_{\R^n} \left|\mathcal{M}_\omega \left(\nabla u\right)(x)\right|^{p(x)} \,dx
$$
can be splitted over the sets of $x\in\R^n$ with
$$
\mathcal{M}_\omega \left(\nabla u\right)(x)\le1 \quad \text{or} \quad \mathcal{M}_\omega \left(\nabla u\right)(x)>1 ,
$$
in order to increase the integrand function by using, respectively, the exponents $p^-$ or $p^+$ and then by extending both the integrals over the entire space $\R^n$. In this way, we get
$$
\int_{\R^n} \left|\mathcal{M}_\omega \left(\nabla u\right)(x)\right|^{p(x)} \,dx \le \int_{\R^n} \left|\mathcal{M}_\omega \left(\nabla u\right)(x)\right|^{p^-} \,dx + \int_{\R^n} \left|\mathcal{M}_\omega \left(\nabla u\right)(x)\right|^{p^+} \,dx, \quad \forall \omega\in\Sf.
$$
By the theory of maximal functions, as treated in \cite{Stein} and adapted to the variable exponent case in \cite{FS-2020}, there exists positive constants $C_{p^\pm}$, depending on $p^\pm$, such that
$$
\begin{aligned}
	\int_{\R^n} \left|\mathcal{M}_\omega \left(\nabla u\right)(x)\right|^{p^-} \,dx & + \int_{\R^n} \left|\mathcal{M}_\omega \left(\nabla u\right)(x)\right|^{p^+} \,dx \\
	& \le C_{p^-} \int_{\R^n} \left|\nabla u(x)\right|^{p^-} \,dx + C_{p^+} \int_{\R^n} \left|\nabla u(x)\right|^{p^+} \,dx\\
	& \le C_{p^\pm} \left(\|\nabla u\|_{L^{p^-}\left(\R^n\right)}^{p^-} + \|\nabla u\|_{L^{p^+}\left(\R^n\right)}^{p^+}\right), \qquad \forall \omega\in\Sf ,
\end{aligned}
$$
where $C_{p^{\pm}}:=\max\left\{C_{p^-},C_{p^+}\right\}$. For this reason, equation \eqref{eq3} becomes
$$
\begin{aligned}
	\Lambda_\delta (u) & \le  \alpha \int_{\R^n} \int_\Sf  \left|\mathcal{M}_\omega \left(\nabla u\right)(x)\right|^{p(x)} \,d\mathcal{H}^{n-1}(\omega) \,dx . \\
	& \le \alpha\, C_{p^\pm} \left(\|\nabla u\|_{L^{p^-}\left(\R^n\right)}^{p^-} + \|\nabla u\|_{L^{p^+}\left(\R^n\right)}^{p^+}\right) \mathcal{H}^{n-1}\left(\Sf\right).
\end{aligned}
$$
The assertion follows taking $C:=\alpha\,C_{p^\pm}\,\mathcal{H}^{n-1}\left(\Sf\right)$.

\item[(b)] By the notion of directional derivative, let's observe that
\begin{equation}
\label{directionalderivative}
\lim_{\delta\to0^+} \frac{|u(x+\delta h\omega)-u(x)|}{\delta} = \left|{\nabla u(x)\cdot\omega}\right|h
\end{equation}
for a.e. $(x,h,\omega)\in\R^n\times \R^+ \times\Sf$. As a consequence, since -- for a.e. $x\in\R^n$ -- $\varphi(x,\cdot)$ is continuous at $0$ and almost everywhere on $(0,+\infty)$, we have
$$
\lim_{\delta\to0^+} \frac{1}{h^{p(x)+1}}\varphi\left(x,\frac{|u(x+\delta h\omega)-u(x)|}{\delta}\right) = \frac{1}{h^{p(x)+1}}\varphi\left(x,\left|{\nabla u(x)\cdot\omega}\right|h\right)
$$
for a.e. $(x,h,\omega)\in\R^n\times \R^+ \times\Sf$. Making the integral of this quantity over the sphere $\Sf$ and respect to $h\in[0,+\infty)$, by replacing $t=\left|{\nabla u(x)\cdot\omega}\right|h$, for a.e. $x\in\R^n$ we get
\begin{equation}
	\label{eq1}
	\begin{aligned}
		\int_\Sf \int_0^{+\infty} \frac{1}{h^{p(x)+1}}&\varphi\left(x,\left|{\nabla u(x)\cdot\omega}\right|h\right)\,dh\,d\mathcal{H}^{n-1}(\omega) \\
		& = \int_\Sf \int_0^{+\infty} \frac{\left|{\nabla u(x)	\cdot\omega}\right|^{p(x)+1}}{t^{p(x)+1}} \varphi(x,t) \frac{1}{\left|{\nabla u(x)	\cdot\omega}\right|}\,dt\,d\mathcal{H}^{n-1}(\omega) \\
		& = \int_0^{+\infty} \varphi(x,t) t^{-(p(x)+1)}\,dt \int_\Sf \left|{\nabla u(x)	\cdot\omega}\right|^{p(x)}\,d\mathcal{H}^{n-1}(\omega).
	\end{aligned}
\end{equation}
Since, for every $V\in\R^n$ and for all $x\in\R^n$,
$$
\int_{\Sf} \left|{V\cdot\omega}\right|^{p(x)} \, d\mathcal{H}^{n-1}(\omega) = |V|^{p(x)} \int_\Sf \left|{\omega\cdot\mathbf{e}}\right|^{p(x)} \, d\mathcal{H}^{n-1}(\omega),
\quad \mathbf{e}\in\Sf.
$$
taken $V=\nabla u(x)$, equation \eqref{eq1} becomes
\begin{equation}
\label{limite}
\begin{aligned}
	\int_\Sf \int_0^{+\infty} \frac{1}{h^{p(x)+1}}&\varphi\left(x,\left|{\nabla u(x)\cdot\omega}\right|h\right)\,dh\,d\mathcal{H}^{n-1}(\omega) \\
	& =  |\nabla u(x)|^{p(x)} \int_\Sf \left|{\omega\cdot\mathbf{e}}\right|^{p(x)} \, d\mathcal{H}^{n-1}(\omega) \int_0^{+\infty} t^{-(p(x)+1)}\varphi(x,t)\,dt \\
	& = |\nabla u(x)|^{p(x)} \gamma_{n,p(x)}\int_0^{+\infty} \varphi(x,t) t^{-(p(x)+1)}\,dt = |\nabla u(x)|^{p(x)},
\end{aligned}
\end{equation}
for a.e. $x\in\R^n$, where we used hypothesis \eqref{Hp4} on $\varphi$. Integrating the last equation with respect to $x$ over the space $\R^n$, we have
$$
	\int_{\R^n} \int_\Sf \int_0^{+\infty} \frac{1}{h^{p(x)+1}}\varphi\left(x,\left|{\nabla u(x)\cdot\omega}\right|h\right)\,dh\,d\mathcal{H}^{n-1}(\omega)\,dx = \int_{\R^n} |\nabla u(x)|^{p(x)} \, dx .
$$
Keeping into account what we have done in the first part of the proof, we are able to apply the dominated convergence theorem, so that
$$
\begin{aligned}
	\lim_{\delta\to0^+}\Lambda_{\delta}(u) & = \lim_{\delta\to0^+} \int_{\R^n}\int_{\Sf}\int_0^{+\infty} \frac{\varphi\left(x,|u(x+\delta h\omega)-u(x)|/\delta\right)}{h^{p(x)+1}}\, dh\,d\mathcal{H}^{n-1}(\omega)\,dx  \\
	& = \int_{\R^n}\int_{\Sf}\int_0^{+\infty} \lim_{\delta\to0^+}  \frac{\varphi\left(x,|u(x+\delta h\omega)-u(x)|/\delta\right)}{h^{p(x)+1}}\, dh\,d\mathcal{H}^{n-1}(\omega)\,dx  \\
	& = \int_{\R^n} \int_\Sf \int_0^{+\infty} \frac{1}{h^{p(x)+1}}\varphi\left(x,\left|{\nabla u(x)\cdot\omega}\right|h\right)\,dh\,d\mathcal{H}^{n-1}(\omega)\,dx = \int_{\R^n} |\nabla u(x)|^{p(x)} \, dx
\end{aligned}
$$
and the assertion follows. \\

We are now ready to discuss the case in which $\Omega\subset\R^n$ is a bounded and smooth domain.
\item[(b)] Let $D\Subset\Omega$ and fix $t>0$ small enough such that
$$
B(x,t)=\left\{y\in\R^n : |x-y| < t\right\} \Subset \Omega , \quad \forall x \in D.
$$
For every $u \in W^{1,p(\cdot)}\left(\Omega\right)$, we have
$$
\Lambda_\delta(u) \ge \int_D \int_{B(x,t)} \frac{\varphi_{\delta} \left(x,|u(x)-u(y)|\right)}{|x-y|^{n+p(x)}}\,dy\,dx.
$$
By using polar coordinates and -- more precisely -- making the change of variables $x-y=\delta h\omega$, for $\omega\in\Sf$, the previously equation becomes
$$
\begin{aligned}
\Lambda_\delta(u) & \ge \int_D \int_0^{t/\delta} \int_{\Sf} \delta^{p(x)} \frac{\varphi \left(x,|u(x+\delta h\omega)-u(x)|\slash\delta\right)}{\left(\delta h\right)^{n+p(x)}}\,\delta^n h^{n-1} dt\,d\mathcal{H}^{n-1}(\omega)\,dx \\
& = \int_D \int_0^{t/\delta} \int_{\Sf} \frac{\varphi \left(x,|u(x+\delta h\omega)-u(x)|\slash\delta\right)}{h^{p(x)+1}}\, dt\,d\mathcal{H}^{n-1}(\omega)\,dx
\end{aligned}
$$
By employing Fatou's lemma, remembering equations \eqref{directionalderivative} and \eqref{limite}, it follows
$$
\begin{aligned}
\liminf_{\delta\to0^+} \Lambda_\delta(u) & \ge \liminf_{\delta\to0^+} \int_D \int_0^{t/\delta} \int_{\Sf} \frac{\varphi \left(x,|u(x+\delta h\omega)-u(x)|\slash\delta\right)}{h^{p(x)+1}}\, dt\,d\mathcal{H}^{n-1}(\omega)\,dx \\
& \ge \int_D  \liminf_{\delta\to0^+} \int_0^{t/\delta} \int_{\Sf} \frac{\varphi \left(x,|u(x+\delta h\omega)-u(x)|\slash\delta\right)}{h^{p(x)+1}}\, dt\,d\mathcal{H}^{n-1}(\omega)\,dx \\
& = \int_D \int_0^{+\infty} \int_{\Sf} \frac{\varphi \left(x,\left|{\nabla u(x) \cdot \omega}\right| h \right)}{h^{p(x)+1}}\, dt\,d\mathcal{H}^{n-1}(\omega)\,dx = \int_D \left|\nabla u(x)\right|^{p(x)}\,dx,
\end{aligned}
$$
so, from the arbitrariness of $D\Subset\Omega$, we have
\begin{equation}
\label{liminf}
\liminf_{\delta\to0^+} \Lambda_\delta(u) \ge \int_\Omega \left|\nabla u(x)\right|^{p(x)}\,dx.
\end{equation}
Now, we want to prove that
$$
\lim_{\delta\to0^+} \Lambda_{\delta}(u) = \int_\Omega |\nabla u(x)|^{p(x)}\,dx, \quad \forall u\in W^{1,p^{\pm}}\left(\Omega\right).
$$
Let $u \in W^{1,p^\pm}\left(\Omega\right)$. Taken a bounded subset $V\subset\R^n$ containing the domain $\Omega$, there exists an extension $\tilde{u} \in W^{1,p^\pm}\left(\R^n\right)$ of $u$, defined over the entire space $\R^n$, such that
\begin{enumerate}
	\item[i.] $\tilde{u}(x) = u(x)$ for a.e. $x\in\Omega$;
\item[ii.] $\tilde{u}$ is compactly supported in $V$, that is $\supt(\tilde{u})\subset V$;
\item[iii.] there exist constants $C_\pm>0$, depending on $p^\pm$, $\Omega$ and $V$, such that
$$
\|\tilde{u}\|_{W^{1,p^+}\left(\R^n\right)} \le C_{+} \|u\|_{W^{1,p^+}\left(\Omega\right)}, \qquad
\|\tilde{u}\|_{W^{1,p^-}\left(\R^n\right)} \le C_{-} \|u\|_{W^{1,p^-}\left(\Omega\right)}
$$
and
$$
\|\tilde{u}\|_{L^{p^+}\left(\R^n\right)} \le C_+ \|u\|_{L^{p^+}\left(\Omega\right)}, \qquad
\|\tilde{u}\|_{L^{p^-}\left(\R^n\right)} \le C_- \|u\|_{L^{p^-}\left(\Omega\right)}.
$$
\end{enumerate}

Notice that it is possible to find an extension over the entire space $W^{1,p^\pm}\left(\R^n\right)$ because, in the classical extension theorem, the approximation sequences involved do not depend on the exponent $p$ of the Sobolev space considered, but only the constant $C$ -- satisfying the norms inequality -- depends on it (See \cite[Part II, Chapter 5]{Evans}).

%In particular, being $u \in W^{1,p^+}\left(\Omega\right)$, from the classical extension theorem, there exists a linear operator $E\in\mathcal{L}\left(W^{1,p^+}\left(\R^n\right),W^{1,p^+}\left(\Omega\right)\right)$ such that:
%\begin{enumerate}
%	\item[i.] $Eu(x) = u(x)$ for a.e. $x\in\Omega$;
%	\item[ii.] $Eu$ is compactly supported in $V$, that is $\supt(Eu)\subset V$;
%	\item[iii.] there exists a constant $C_+>0$, depending on $p^+$, $\Omega$ and $V$, such that
%	$$
%		\|Eu\|_{W^{1,p^+}\left(\R^n\right)} \le C_+ \|u\|_{W^{1,p^+}\left(\Omega\right)}
%	$$
%	and
%	$$
%		\|Eu\|_{L^{p^+}\left(\R^n\right)} \le C_+ \|u\|_{L^{p^+}\left(\Omega\right)}.
%	$$
%\end{enumerate}
%Let's observe that $\tilde{u}:=Eu$ satisfies the properties stated, since $V$ is bounded and $p^+\slash p^- \ge 1$. In fact, taken a function $f\in L^{p^+}\left(V\right)$, by applying H\"{o}lder's inequality with exponent $p^+\slash p^-$ and its conjugate $\left(p^+\slash p^-\right)'=p^+ \slash \left(p^+-p^-\right)$, we have
%$$
%\int_V |f(x)|^{p^-} dx \le \left(\mathcal{L}^n\left(V\right)\right)^{\frac{p^+-p^-}{p^+}} \left(\int_V |f(x)|^{p^+} \right)^{\frac{p^-}{p^+}} < +\infty ,
%$$
%so that $f \in L^{p^{\pm}} \left(V\right) := L^{p^+}\left(V\right) \cap L^{p^-}\left(V\right)$. For $f=\tilde{u}$ and $f=\nabla \tilde{u}$, being $\supt\left(\tilde{u}\right) \subset V$, this means $\tilde{u} \in W^{1,p^{\pm}} \left(\R^n\right)$, as desired.

Once we have extended $u$ through $\tilde{u}$, we have
$$
\Lambda_\delta(u) \le \int_{\Omega}dx \int_{\R^n} \frac{\varphi_{\delta} \left(x,|\tilde{u}(x)-\tilde{u}(y)|\right)}{|x-y|^{n+p(x)}}\,dy,
$$
so, by making the superior limit as $\delta\to0^+$ -- being now in the $\R^n$ case --, we have
$$
\begin{aligned}
\limsup_{\delta\to0^+} \Lambda_\delta(u) & \le \limsup_{\delta\to0^+} \int_{\Omega}dx \int_{\R^n} \frac{\varphi_{\delta} \left(x,|\tilde{u}(x)-\tilde{u}(y)|\right)}{|x-y|^{n+p(x)}}\,dy \\
& = \int_{\Omega}dx \lim_{\delta\to0^+} \int_{\R^n} \frac{\varphi_{\delta} \left(x,|\tilde{u}(x)-\tilde{u}(y)|\right)}{|x-y|^{n+p(x)}}\,dy \\
& = \int_{\Omega}|\nabla \tilde{u}(x)|^{p(x)}\,dx = \int_{\Omega}|\nabla u(x)|^{p(x)}\,dx .
\end{aligned}
$$
Remembering equation \eqref{liminf},
$$
\int_\Omega \left|\nabla u(x)\right|^{p(x)}\,dx \le \liminf_{\delta\to0^+} \Lambda_\delta(u) \le \limsup_{\delta\to0^+} \Lambda_\delta(u) \le \int_{\Omega}|\nabla u(x)|^{p(x)}\,dx,
$$
so that the limit exists and is
$$
\lim_{\delta\to0^+} \Lambda_\delta(u) =  \int_{\Omega}|\nabla u(x)|^{p(x)}\,dx,
$$
as we wanted to prove.

\item[(a)] Finally, let's show that
$$
\Lambda_{\delta}(u) \le C\left(\|\nabla u\|_{L^{p^+}\left(\Omega\right)}^{p^+} + \|\nabla u\|_{L^{p^-}\left(\Omega\right)}^{p^-}\right) .
$$
Up to replacing  $u(x)$ with $u(x)-\int_\Omega u(\xi) \, d\xi$,
we may assume that $\int_\Omega u(x) \, dx = 0$.
Then, since $\Omega$ is smooth -- as shown in \cite[Chapter 9]{Brezis} -- there is an extension $\tilde{u} \in W^{1,p^\pm}\left(\R^n\right)$ such that
$$
\int_{\R^n} |\nabla \tilde{u}(x) |^{p^+} \, dx \le C_{p^+} \int_{\R^n} |\nabla u(x) |^{p^+} \, dx \qquad \text{and} \qquad \int_{\R^n} |\nabla \tilde{u}(x) |^{p^-} \, dx \le C_{p^-} \int_{\R^n} |\nabla u(x) |^{p^-} \, dx .
$$
Being
$$
\begin{aligned}
\Lambda_\delta \left(u,\Omega\right) \le \Lambda_\delta \left(\tilde{u},\R^n\right) & \le \tilde{C} \left(\|\nabla \tilde{u}\|_{L^{p^+}\left(\Omega\right)}^{p^+} + \|\nabla \tilde{u}\|_{L^{p^-}\left(\Omega\right)}^{p^-}\right) \\
& \le \tilde{C} \left(C_{p^-}\|\nabla u\|_{L^{p^+}\left(\Omega\right)}^{p^+} + C_{p^+}\|\nabla u\|_{L^{p^-}\left(\Omega\right)}^{p^-}\right),
\end{aligned}
$$
the assertion follows.
%\textcolor{blue}{Let's now assume that $\varphi(x,\cdot)$ is a non-decreasing function for a.e. $x\in\R^n$. If we take a function $u\in L^{p(\cdot)}\left(\Omega\right)$ such that
%$$
%\liminf_{\delta\to0^+} \Lambda_{\delta}(u) < +\infty ,
%$$
%by applying Fatou's Lemma and arguing as before, we have
%$$
%\int_\Omega |\nabla u(x)|^{p(x)}\,dx \le \liminf_{\delta\to0^+} \int_\Omega \int_\Omega \frac{\varphi_{\delta} \left(x,|u(x)-u(y)|\right)}{|x-y|^{n+p(x)}}\,dx\,dy = \liminf_{\delta\to0^+} \Lambda_{\delta}(u) < + \infty,
%$$
%so $u\in W^{1,p(\cdot)}\left(\Omega\right)$.}
%\textcolor{red}{Perché uno deve prendere l'ipotesi di monotonia sulla $\varphi$??? Affinche' sia lecito il limite interno all'integrale, sfruttando Fatou?}
%
%\textcolor{blue}{Problema di esistenza del gradiente debole... Legato a qualche costruzione che sfrutta la monotonia?}
\end{proof}

%\begin{remark}
%The previous Theorem provides a membership criterion to the variable exponent Sobolev space $W^{1,p(\cdot)}\left(\Omega\right)$ -- with $1<p^-\le p^+<+\infty$ -- as
%$$
%\left\{u\in L^{p(\cdot)} \left(\Omega\right) : \liminf_{\delta\to0^+} \Lambda_{\delta}(u) < +\infty\right\} \subset W^{1,p(\cdot)}\left(\Omega\right),
%$$
%for some $\varphi : \R^n\times [0,+\infty) \to [0,+\infty)$ satisfying all the hypothesis stated at the beginning of this paper.
%\end{remark}
\noindent
We assume now that
\begin{equation}
	\label{Hp3}
	\varphi(x,\cdot) \text{ is a non-decreasing function},
\end{equation}
for a.e. $x\in\R^n$.

\begin{proposition}
	\label{suffcond1}
Let $1 < p^- \le p^+ < +\infty$ and let $u \in L^{p(\cdot)}\left(\Omega\right) \cap C^2\left(\Omega\right)$. Suppose that $\varphi$ satisfies the following properties
\begin{equation}
	\label{phi1}
	\exists \alpha,a>0 : \quad \alpha  t^{{p^+}+1}\leq \varphi(x,t) \le at^{p(x)+1} ,\quad \forall t \in [0,1],\,\, \text{for a.e.}\,\, x\in\R^n,
\end{equation}
\begin{equation}
	\label{phi2}
	\exists \beta,b>0 : \quad \beta\leq \varphi(x,t) \le b ,\quad \forall t \in \R^+ ,\,\, \text{for a.e.}\,\, x\in\R^n,
\end{equation}
and also hypothesis  \eqref{Hp4} and \eqref{Hp3}. Then
$$
		\limsup_{\delta \to 0^+} \Lambda_{\delta}(u) \ge \int_{\Omega} \left | \nabla u(x) \right |^{p(x)} \, dx.
		\label{prop2eq1}
$$
%\textcolor{red}{estendibile a solo $u \in L^{p(\cdot)}(\Omega)$??}
In particular, $u\in W^{1,p(\cdot)}(\Omega)$ whenever $\sup_{\delta>0} \Lambda_\delta(u)<+\infty$.
\end{proposition}
\begin{proof}
It is sufficient to treat the case
$$
		F:=\limsup_{\delta \to 0^+} \Lambda_{\delta}(u) < +\infty ,
		\label{prop2eq2}
$$
otherwise the inequality is obvious.
First of all, let us suppose further that $u \in L^{\infty}(\Omega)$ and let
\begin{equation}
\label{prop2eq3}
A := 2 \|u\|_{L^{\infty}\left(\Omega\right)}.
\end{equation}
Taken $\delta_{0}\in(0,1)$ and fixed $\varepsilon \in (0,1\slash2)$, we have
\begin{equation}
\label{prop2eq4}
		T(\varepsilon, \delta_{0}) := \int_{0}^{\delta_{0}} \varepsilon \delta^{\varepsilon -1} \Lambda_{\delta}(u) \, d\delta =  \int_{0}^{\delta_{0}} \varepsilon \delta^{\varepsilon -1} \,d\delta \int_{\Omega} \int_{\Omega}  \frac{\delta^{p(x)} \varphi \left (x, \left | u(x)-u(y) \right | / \delta \right )}{\left | x-y \right |^{n+p(x)}} \,dx\,dy.
\end{equation}
By making the change of variables $t = \left| u(x)-u(y) \right| \slash \delta$, we get
	$$ 
	\begin{aligned}
	& T(\varepsilon, \delta_{0}) \\
	& = \int_{\Omega} \int_{\Omega} \int_{+\infty}^{\frac{\left | u(x)-u(y) \right |}{\delta_{0}}} \frac{\varepsilon \left | u(x)-u(y) \right | ^{\varepsilon-1}}{t^{\varepsilon-1}} \frac{\left | u(x)-u(y) \right |^{p(x)}}{t^{p(x)}}  \left ( - \frac{\left | u(x)-u(y) \right |}{t^2} \right ) \frac{\varphi(x,t)}{|x-y| ^{n+p(x)}} \, dt \, dx \, dy \\
	& = \int_{\Omega} \int_{\Omega}  \frac{\varepsilon \left | u(x)-u(y) \right | ^{p(x)+ \varepsilon}}{\left | x-y \right |^{n+p(x)}} \,dx\,dy \int_{\left | u(x)-u(y) \right | / \delta_{0}}^{+\infty} \varphi(x,t) t^{-1-p(x)-\varepsilon} \,dt .
	\end{aligned}
	$$
It follows that
	\begin{equation}
	\begin{aligned}
		\label{eq_58}
	T(\varepsilon, \delta_{0}) & := \int_{\Omega} \int_{\Omega}  \frac{\varepsilon \left | u(x)-u(y) \right | ^{p(x)+ \varepsilon}}{\left | x-y \right |^{n+p(x)}} \,dx\,dy \int_{\left | u(x)-u(y) \right | / \delta_{0}}^{+\infty} \varphi(x,t) t^{-1-p(x)-\varepsilon} \,dt \\
	& \ge \underset{ \left | u(x)-u(y) \right | < \delta_{0}^2}{\int_\Omega \int_\Omega} \frac{\varepsilon \left | u(x)-u(y) \right | ^{p(x)+ \varepsilon}}{\left | x-y \right |^{n+p(x)}} \,dx\,dy \int_{\left | u(x)-u(y) \right | / \delta_{0}}^{+\infty} \varphi(x,t) t^{-1-p(x)-\varepsilon} \,dt \\
	& \ge c_0 \underset{ \left | u(x)-u(y) \right | < \delta_{0}^2 }{\int_\Omega \int_\Omega} \frac{\varepsilon \left | u(x)-u(y) \right | ^{p(x)+ \varepsilon}}{\left | x-y \right |^{n+p(x)}} \,dx\,dy
	\end{aligned}
	\end{equation}
where, using \eqref{phi1} and \eqref{phi2}, we introduced $c_0>0$ (depending only on $p$ and $\beta$)
through the following inequalities (recall that 
$|u(x)-u(y)|/\delta_{0} < \delta_{0}$ on the integration set)
	$$
	\begin{aligned}
	\int_{\left | u(x)-u(y) \right | / \delta_{0}}^{+\infty} \varphi(x,t) t^{-1-p(x)-\varepsilon} \,dt & \ge \int_{\delta_0}^1 \alpha t^{1+p^+} t^{-1-p(x)-\varepsilon}\,dt + \int_1^{+\infty} \beta t^{-1-p(x)-\varepsilon}\,dt \\
	& \ge \alpha \int_{\delta_0}^1 t^{p^+ - p^- -\varepsilon}\,dt + \beta \int_1^{+\infty} t^{-1 - p^+ - \varepsilon}\,dt
	\ge \frac{\beta}{2p^+ }= : c_0 
	\end{aligned}
	$$
Now equation \eqref{eq_58} can be written as
$$
T(\varepsilon, \delta_{0}) \ge c_0 {\int_\Omega \int_\Omega} \frac{\varepsilon \left | u(x)-u(y) \right | ^{p(x)+ \varepsilon}}{\left | x-y \right |^{n+p(x)}} \,dx\,dy - c_0 \underset{ \left | u(x)-u(y) \right | \ge {\delta_{0}}^2 }{\int_\Omega \int_\Omega} \frac{\varepsilon A^{p(x)+\varepsilon}}{|x-y|^{n+p(x)}}\,dx\,dy ,
$$
remembering that we posed $A := 2 \|u\|_{L^\infty\left(\Omega\right)}$.
Let $\tau > 0$ and $\delta_0$ small enough such that
\begin{equation}
	\label{stima_c}
c_0 \ge (1-\tau) \left( \int_\Sf \left|\omega\cdot  \mathbf{e}\right|^{p^+} \, d\mathcal{H}^{n-1}(\omega) \right)^{-1}
\end{equation}
and
\begin{equation}
	\label{Ftau}
\Lambda_{\delta}(u) \le F + \tau, \qquad \forall \delta \in \left(0,\delta_0\right). 
\end{equation}
We have
\begin{equation}
	\label{integrale_finito}
\underset{\left|u(x)-u(y)\right|\ge\gamma}{\int_\Omega \int_\Omega} \frac{1}{|x-y|^{n+p(x)}}\,dx\,dy < +\infty , \qquad \forall \gamma > 0.
\end{equation}
In fact, let's fix $t_0 > 0$ such that $\inf_\Omega\varphi\left(x,t_0\right) > 0$ and observe that
$$
\begin{aligned}
\Lambda_{\delta}(u) \ge \underset{\left|u(x)-u(y)\right|\ge\gamma}{\int_\Omega \int_\Omega} & \frac{\varphi_{\delta}\left(x, \left|u(x)-u(y)\right|\right)}{|x-y|^{n+p(x)}}\,dx\,dy\\
& \ge \delta^{p^+} \inf_{x\in\Omega} \varphi\left(x, \gamma/\delta\right) \underset{\left|u(x)-u(y)\right|\ge\gamma}{\int_\Omega \int_\Omega} \frac{1}{|x-y|^{n+p(x)}}\,dx\,dy.
\end{aligned}
$$
Taken $0<\delta <\min\left\{\delta_0,\gamma/t_0\right\}$ and applying equation \eqref{Ftau}, we get
$$
F+\tau \ge \delta^{p^+} \inf_{x\in\Omega} \varphi\left(x, \gamma/\delta\right) \underset{\left|u(x)-u(y)\right|\ge\gamma}{\int_\Omega \int_\Omega} \frac{1}{|x-y|^{n+p(x)}}\,dx\,dy ,
$$
so \eqref{integrale_finito} follows since $F<+\infty$. Thanks to this equation, we have
$$
\begin{aligned}
0 \le \lim_{\varepsilon\to0^+} c_0 & \underset{ \left | u(x)-u(y) \right | \ge \delta_{0}^2 }{\int_\Omega \int_\Omega}  \frac{\varepsilon A^{p(x)+\varepsilon}}{|x-y|^{n+p(x)}}\,dx\,dy\\
&  \le \lim_{\varepsilon\to0^+} \varepsilon c_0 \max\left\{A^{p^++\varepsilon},A^{p^-+\varepsilon}\right\} \underset{ \left | u(x)-u(y)\right| \ge \delta_{0}^2 }{\int_\Omega \int_\Omega} \frac{1}{|x-y|^{n+p(x)}}\,dx\,dy = 0 ,
\end{aligned}
$$
so that
\begin{equation}
\begin{aligned}
	\label{limite2}
\liminf_{\varepsilon\to0^+} T(\varepsilon, \delta_{0})  \ge  \liminf_{\varepsilon\to0^+} & \, c_0 {\int_\Omega \int_\Omega} \frac{\varepsilon \left | u(x)-u(y) \right | ^{p(x)+ \varepsilon}}{\left | x-y \right |^{n+p(x)}} \,dx\,dy \\
& - \lim_{\varepsilon\to0^+}c_0 \underset{ \left | u(x)-u(y) \right | \ge {\delta_{0}}^2 }{\int_\Omega \int_\Omega} \frac{\varepsilon A^{p(x)+\varepsilon}}{|x-y|^{n+p(x)}}\,dx\,dy \\
& = \liminf_{\varepsilon\to0^+} c_0 {\int_\Omega \int_\Omega} \frac{\varepsilon \left | u(x)-u(y) \right | ^{p(x)+ \varepsilon}}{\left | x-y \right |^{n+p(x)}} \,dx\,dy .
\end{aligned}
\end{equation}
Remembering -- as shown in \cite{FS-2020} -- that, for every $u \in L^{p(\cdot)}\left(\Omega\right) \cap C^{2}\left(\Omega\right)$ it holds
$$
		\int_{\Omega} \gamma_{n,p(x)} \left|\nabla u(x)\right|^{p(x)} \, dx \le \liminf_{\varepsilon\to0^+} \int_{\Omega} \int_{\Omega} \frac{\varepsilon\left|u(x)-u(y)\right|^{p(x)+\varepsilon}}{|x-y|^{n+p(x)}} \, dx \, dy
$$
 and -- being $\gamma_{n,p^+} \le \gamma_{n,p(x)}$ for a.e. $x\in\Omega$ -- we have
\begin{equation}
			\label{prop2eq14}
		\int_{\mathbb{S}^{n-1}} \left|\,\omega\cdot  \mathbf{e}\,\right|^{p^{+}}\,d\mathcal{H}^{n-1}(\omega) \int_{\Omega} \left|\nabla u(x)\right|^{p(x)} \, dx \le \liminf_{\varepsilon\to0^+} \int_{\Omega} \int_{\Omega} \frac{\varepsilon\left|u(x)-u(y)\right|^{p(x)+\varepsilon}}{|x-y|^{n+p(x)}} \, dx \, dy.
\end{equation}
We are now ready to conclude. By combining \eqref{limite2}, \eqref{stima_c} and \eqref{prop2eq14} , we get
	\begin{equation}
		\begin{aligned}
		\liminf_{\varepsilon \to 0^+} T(\varepsilon, \delta_{0}) & \ge \liminf_{\varepsilon\to0^+} c_0 {\int_\Omega \int_\Omega} \frac{\varepsilon \left | u(x)-u(y) \right | ^{p(x)+ \varepsilon}}{\left | x-y \right |^{n+p(x)}} \,dx\,dy \\
		&  \ge (1-\tau) \left(\int_\Sf \left|\omega\cdot  \mathbf{e}\right|^{p^+} \, d\mathcal{H}^{n-1}(\omega)\right)^{-1} \liminf_{\varepsilon\to0^+} {\int_\Omega \int_\Omega} \frac{\varepsilon \left | u(x)-u(y) \right | ^{p(x)+ \varepsilon}}{\left | x-y \right |^{n+p(x)}} \,dx\,dy \\
		& \ge (1-\tau)  \int_{\Omega} \left | \nabla u(x) \right |^{p(x)}\,dx .
		 \end{aligned}
		\label{prop2eq15}
	\end{equation}
From the definition of $T\left(\varepsilon,\delta_0\right)$ stated in \eqref{prop2eq4} and equation \eqref{Ftau}, we have
	$$
	T(\varepsilon, \delta_{0})   := \int_{0}^{\delta_{0}} \varepsilon \delta^{\varepsilon -1} \Lambda_{\delta}(u) \, d\delta \le \int_{0}^{\delta_{0}} \varepsilon \delta^{\varepsilon -1} (F+ \tau) \,d\delta = (F+\tau) \delta_{0}^{\varepsilon}
	$$
	so it follows that
	\begin{equation}
		\limsup_{\varepsilon \to 0^+} T(\varepsilon, \delta_{0}) \le F + \tau .
		\label{prop2eq16}
	\end{equation}
	Thanks to equations \eqref{prop2eq15} and \eqref{prop2eq16}, we have
	$$
	\limsup_{\delta \to 0^+} \Lambda_{\delta}(u) + \tau \ge \limsup_{\varepsilon \to 0^+} T(\varepsilon, \delta_{0}) \ge \liminf_{\varepsilon \to 0^+} T(\varepsilon, \delta_{0}) \ge (1-\tau)  \int_{\Omega} \left | \nabla u(x) \right |^{p(x)}\,dx,
	$$
	so for the arbitrariness of $\tau>0$, we finally get
	$$
	\limsup_{\delta \to 0^+} \Lambda_{\delta}(u) \ge \int_{\Omega} \left | \nabla u(x) \right |^{p(x)} \,dx
	$$
	for all $u \in L^{p(\cdot)}\left(\Omega\right) \cap C^{2}\left(\Omega\right)$ bounded. 
	
	If instead $u$ is not bounded, for $M>0$
	let ${\mathcal T}_M\in C^\infty(\R)$ be such that ${\mathcal T}_M(s)=s$ if $|s|\leq M$ and ${\mathcal T}_M(s)=M+1$ if $|s|\geq M+1$
	and denote $u_M := {\mathcal T}_M(u).$
	Then, we have 
	%\footnote{From H.-M.: Please check this point again! The constant $\gamma$ is taken out in %comparison with the previous version.}
	\begin{equation*}
		%\label{R2-claim}
		|{\mathcal T}_M(s_1) - {\mathcal T}_M(s_2)| \le  |s_1 - s_2|,\quad 
		\mbox{ for all } s_1, s_2 \in \R. 
	\end{equation*} 
	Hence $\Lambda_\delta(u_M)
	\le  \Lambda_{\delta}(u)$ by the monotonicity of $\varphi$ and the assertion follows from the previous case
	by the arbitrariness of $M$ since $u_M\to u$ in $W^{1,p(\cdot)}(\Omega)$ by dominated convergence.
%	\textcolor{red}{From a density argument, the assertion follows also in the case $u\in L^{p(\cdot)}\left(\Omega\right)$???}
	
%	\textcolor{red}{??}Nel caso generale con $ u \in L^{p(\cdot)}(\Omega) \cap C^{2}(\Omega)$ si procede ponendo $A>0$ e definendo 
%$$
%		T_{A}(s) := \begin{cases} s & \mathrm{per} \ |s| \ge A \\  A & \mathrm{per} \ s > A \\  -A & \mathrm{per} \ s < -A \end{cases} 
%$$
%	e 
%	$$u_{A} := T_{A}(u).$$
	
%	\textcolor{blue}{La troncatura sta sicuramente in $L^\infty$, ma chi mi dice che questa troncata e' $C^2$???}
	
%	Poichè $\varphi$ è non decrescente si ha
%	$$\Lambda_{\delta}(u_{A}) \le \Lambda_{\delta}(u)$$
%	da cui segue che 
%	$$ \int_{\Omega} \left | \nabla u_{A}(x) \right |^{p(x)}\,dx \le \limsup_{\delta \to 0^+} \Lambda_{\delta}\left(u_{A}\right) \le \limsup_{\delta \to 0^+} \Lambda_{\delta}(u) $$
%	facendo poi tendere $A$ all'infinito si ottiene la tesi
%	$$\int_{\Omega} \left | \nabla u(x) \right |^{p(x)}\,dx \le \limsup_{\delta \to 0^+} \Lambda_{\delta}(u)$$
\end{proof}

%%%%%%%%%%%%%%%%%%%%%%%%%%%%%%%%%%%%%%%%%%%%%%%%%

\subsection{The case $p^- =1$} We have the following main result

\begin{theorem}
	\label{prop11}
	Let $1\leq p^-\le p^+<+\infty$. Assume \eqref{Hp4}, \eqref{Hp1}, \eqref{Hp2} and \eqref{Hp3}. Then 
$$
		\lim_{\delta \to 0^+} \Lambda_{\delta} (u) = \int_{\Omega} \left | \nabla u(x) \right |^{p(x)}\,dx 
		\label{prop11eq1}
$$
	for every $u\in C^1\left(\bar{\Omega}\right)$ or -- in the case $\Omega=\R^n$ -- for every $u \in C^1_c\left(\R^n\right)$. Also,  we have
	\begin{equation}
		\liminf_{\delta \to 0^+} \Lambda_{\delta} (u) \ge \int_{\Omega} \left | \nabla u(x) \right |^{p(x)}\,dx.
		\label{prop11eq2}
	\end{equation} 
	Hence, $u \in W^{1,p(\cdot)}(\Omega)$ provided that
	$\liminf_{\delta \to 0^+} \Lambda_{\delta} (u) <+\infty$.
\end{theorem}
\begin{proof}
	Let us first consider the case $\Omega=\R^n$ and $u \in C_{c}^{1}(\R^n)$. Taken $M>1$ such that $u(x)=0$ if $|x| \ge M-1$, we can split $\Lambda_{\delta} (u)$ as follows
	\begin{equation}
		\label{proposition11eqz}
		\Lambda_{\delta} (u) = \int_{|x|>M}dx \int_{\R^n} \frac {\varphi_{\delta} \left ( x, \left | u(x)-u(y) \right | \right )}{\left | x-y \right |^{n+p(x)}} \,dy +  \int_{|x| \le M}dx \int_{\R^n} \frac {\varphi_{\delta} \left ( x, \left | u(x)-u(y) \right | \right )}{\left | x-y \right |^{n+p(x)}} \,dy.
	\end{equation}
	Since $\varphi$ is bounded and $\varphi(x,0) = 0$ for a.e. $x\in\R^n$, being
	$$
	\int_{|x|>M}dx \int_{|y|<M-1} \frac {1}{\left | x-y \right |^{n+p(x)}} \,dy < +\infty,
	$$
	from the choice of $M$, we have for all $\delta>0$
$$
		\begin{aligned}
			\int_{|x|>M} dx \int_{\R^n} \frac {\varphi_{\delta} \left ( x, \left | u(x)-u(y) \right | \right )}{\left | x-y \right |^{n+p(x)}} \, dy & = \int_{|x|>M} \delta^{p(x)} dx \int_{|y|<M-1} \frac {\varphi \left ( x, \left | u(x)-u(y) \right |\slash \delta \right )}{\left | x-y \right |^{n+p(x)}} \, dy \\
			& \le C\,\delta \int_{|x|>M}dx \int_{|y|<M-1} \frac {1}{\left | x-y \right |^{n+p(x)}} \, dy , 
		\end{aligned}
$$
	so, letting $\delta\to0^+$,
	$$
		\lim_{\delta\to0^+}\int_{|x|>M} dx \int_{\R^n} \frac {\varphi_{\delta} \left ( x, \left | u(x)-u(y) \right | \right )}{\left | x-y \right |^{n+p(x)}} \, dy = 0.
	$$
	Let's now consider the second integral of \eqref{proposition11eqz}. Through the change of variables $z = x - y$ and by using polar coordinates for $z$, we get
$$
		\begin{aligned}
			\int_{|x| \le M}dx & \int_{\R^n} \frac {\varphi_{\delta} \left ( x,\left | u(x)-u(y) \right | \right )}{\left | x-y \right |^{n+p(x)}}\, dy \\
			& =  \int_{|x| \le M}dx \int_{0}^{+\infty} dh \int_{\Sf} \frac {\varphi_{\delta} \left (x, \left | u(x+h\omega)-u(x) \right | \right )}{h^{p(x)+1}} \,d\mathcal{H}^{n-1}(\omega) \\
			& = \int_{|x| \le M}\delta^{p(x)} \, dx \int_{0}^{+\infty} dh \int_{\Sf} \frac {\varphi \left (x, \left | u(x+h\omega)-u(x) \right | / \delta \right )}{h^{p(x)+1}} \,d\mathcal{H}^{n-1}(\omega) .
		\end{aligned}
$$
	By setting $h = \delta \tilde{h}$ and relabeling $\tilde{h}$ as $h$, we have
	$$
	\begin{aligned}
		\int_{|x| \le M}\delta^{p(x)}\,dx \int_{0}^{+\infty} dh & \int_{\Sf} \frac {\varphi \left (x, \left | u(x+h\omega)-u(x) \right | / \delta \right )}{h^{p(x)+1}} \,d\mathcal{H}^{n-1}(\omega) \\
		& = \int_{|x| \le M}dx \int_{0}^{+\infty} dh \int_{\Sf} \frac {\varphi \left (x, \left | u(x+\delta h\omega)-u(x) \right | / \delta \right )}{h^{p(x)+1}} \,d\mathcal{H}^{n-1}(\omega).
	\end{aligned}
	$$
	Since
	\begin{equation}
		\lim_{\delta \to 0^+} \frac{\left | u(x+\delta h \omega) - u(x) \right |}{\delta} = \left | {\nabla u(x) \cdot \omega} \right | h
		\label{prop1eq81}
	\end{equation}
	for a.e. $(x, h, \omega) \in \R^n \times \R^+ \times \Sf$, remembering that $\varphi(x,\cdot)$ is continuous at $0$ almost everywhere on $(0,+\infty)$, it follows
	$$
	\lim_{\delta\to0^+} \frac{1}{h^{p(x)+1}}\varphi\left(x,\frac{|u(x+\delta h\omega)-u(x)|}{\delta}\right) = \frac{1}{h^{p(x)+1}}\varphi\left(x,\left|{\nabla u(x)\cdot\omega}\right|h\right)
	$$
	for a.e. $(x,h,\omega)\in\R^n\times \R^+ \times\Sf$. 
	By integrating over the sphere $\Sf$ and respect to $h\in[0,+\infty)$, replacing $t=\left|{\nabla u(x)\cdot\omega}\right|h$, for a.e. $x\in\R^n$ we get
	\begin{equation}
		\label{prop1eq101}
		\begin{aligned}
			\int_\Sf \int_0^{+\infty} \frac{1}{h^{p(x)+1}}&\varphi\left(x,\left|{\nabla u(x)\cdot\omega}\right|h\right)\,dh\,d\mathcal{H}^{n-1}(\omega) \\
			& = |\nabla u(x)|^{p(x)} \gamma_{n,p(x)}\int_0^{+\infty} \varphi(x,t) t^{-(p(x)+1)}\,dt = |\nabla u(x)|^{p(x)}
		\end{aligned}
	\end{equation}
	where we have used again that, for $V \in \R^n$ and any $\mathbf{e}\in\Sf$, we have
	$$
	\int_{\Sf} \left|{V\cdot\omega}\right|^{p(x)} \, d\mathcal{H}^{n-1}(\omega) = |V|^{p(x)} \int_\Sf \left|{\omega\cdot \mathbf{e}}\right|^{p(x)} \, d\mathcal{H}^{n-1}(\omega) ,\quad x \in \R^n,
	$$
	and the normalization condition on $\varphi$.
	As a consequence, we have
	$$
	\int_{|x| \le M}dx \int_\Sf \int_0^{+\infty} \frac{1}{h^{p(x)+1}}\varphi\left(x,\left|{\nabla u(x)\cdot\omega}\right|h\right)\,dh\,d\mathcal{H}^{n-1}(\omega) = \int_{|x| \le M} |\nabla u(x)|^{p(x)} \,dx.
	$$
	Setting $\tilde{\varphi}: \R^n \times \left [0, +\infty \right ) \rightarrow \mathbb{R}$ as
	$$
	\tilde{\varphi}(x,t) = 
	\begin{cases}
		at^{p(x)+1}, & \ \mathrm{for} \ 0 \le t \le 1 \\ 
		b, & \ \mathrm{for} \ t>1 
	\end{cases},
	$$
	then $\tilde{\varphi}(x,\cdot)$ is a non-decreasing function for a.e. $x \in \R^n$,
	$$
	\varphi(x,t) \le \tilde{\varphi}(x,t) \qquad t \geq 0, \,\,\text{for a.e.}\,\, x\in\R^n
	$$
	and
	\begin{equation}
		\int_{0}^{+\infty} \tilde{\varphi}(t)t^{-(p(x)+1)}dt < +\infty .
		\label{prop1eq131}
	\end{equation}
	Since $ u \in C_{c}^{1}(\R^n)$, we have
$$
\left | u(x+\delta h \omega)-u(x) \right |/\delta \le Ch \quad \forall (x,h,\omega) \in \R^n \times \left [ 0, +\infty \right ) \times \Sf
$$
	for some constant $C\ge0$. Then 
	$$
	\varphi\left(x,|u(x+\delta h\omega)-u(x)|/\delta\right)/h^{p(x)+1}
	$$
	is dominated by $\tilde{\varphi}(x,Ch)/h^{p(x)+1}$, which is summable as
$$
		\int_{|x| \le M} dx \int_{0}^{+\infty} dh \int_{\Sf}  \frac{1}{h^{p(x)+1}} \tilde{\varphi} \left ( x, Ch \right ) d\mathcal{H}^{n-1}(\omega) < +\infty.
$$
	We are now able to apply the dominated convergence theorem, getting
$$
		\begin{aligned}
			\lim_{\delta \to 0^+}\Lambda_{\delta} (u) & = \lim_{\delta \to 0^+}\int_{|x| \le M}dx \int_{0}^{+\infty} dh \int_{\Sf} \frac {\varphi \left (x, \left | u(x+\delta h\omega)-u(x) \right | / \delta \right )}{h^{p(x)+1}} \, d\mathcal{H}^{n-1}(\omega) \\
			& = \int_{|x| \le M}dx \int_{0}^{+\infty} dh \int_{\Sf} \frac {\varphi \left (x, \left| {\nabla u(x)\cdot\omega} \right| h \right )}{h^{p(x)+1}} \, d\mathcal{H}^{n-1}(\omega) \\
			& = \int_{|x| \le M} \left | \nabla u(x) \right |^{p(x)} \,dx = \int_{\R^n} \left | \nabla u(x) \right |^{p(x)} dx. 
		\end{aligned}
$$
The proof of (\ref{prop11eq2}) is very similar, as a consequence of equations (\ref{prop1eq81}), (\ref{prop1eq101}) and Fatou's lemma. 	

	Let's now suppose that $\Omega \subset \R^n$ is a smooth and bounded domain and let $u\in C^1\left(\bar{\Omega}\right)$. If we take $D \Subset \Omega$ and fix $t>0$ small enough such that
	$$
		B(x,t)=\left \{ y \in \R^n : \left | y-x \right | < t \right \} \Subset \Omega, \quad \forall x \in D,
	$$
	than we have
	$
		\Lambda_{\delta}(u) = A_{\delta}+B_{\delta}+C_{\delta},
	$
	where
	$$
	\begin{aligned}
	& A_{\delta} := \int_{D}  \int_{B(x,t)} \frac{ \varphi_{\delta} \left (x, \left | u(x) - u(y) \right | \right )}{\left | x-y \right | ^{n+p(x)}}\, dy \, dx, \\
	& B_{\delta} :=  \int_{D}  \int_{\Omega \setminus B(x,t)} \frac{ \varphi_{\delta} \left (x, \left | u(x) - u(y) \right | \right )}{\left | x-y \right | ^{n+p(x)}} \, dy \, dx, \\
	& C_{\delta} := \int_{\Omega \setminus D} \int_{\Omega} \frac{ \varphi_{\delta} \left ( x,\left | u(x) - u(y) \right | \right )}{\left | x-y \right | ^{n+p(x)}} \, dy \, dx.
	\end{aligned}
	$$
	By the change of variables $x-y=\delta h \omega$ and dominated convergence theorem, we get
	\begin{equation}
	\label{stima_A}
	\begin{aligned}
		\lim_{\delta\to0^+} A_\delta & = \lim_{\delta\to0^+} \int_{D} \int_{B(x,t)} \frac{ \varphi_{\delta} \left ( x, \left | u(x) - u(y) \right | \right )}{\left | x-y \right | ^{n+p(x)}}\, dy \, dx \\
		& = \int_{D} \lim_{\delta\to0^+} \int_{0}^{t/ \delta} \int_{\Sf} \frac{ \varphi \left ( x, \left | u(x+\delta h \omega) - u(x) \right | / \delta \right )}{h^{p(x)+1}} \, d\mathcal{H}^{n-1}(\omega) \, dh \, dx \\
		& =  \int_{D} \int_{0}^{+\infty} \int_{\Sf} \frac{ \varphi \left ( x, \left |\nabla u(x)  \cdot \omega \right | h\right )}{h^{p(x)+1}}\ d\mathcal{H}^{n-1}(\omega) \, dh \, dx= \int_{D} | \nabla u(x)|^{p(x)} \, dx.
	\end{aligned}
	\end{equation}
	Now we claim that
	\begin{equation}
		\label{stime_delta}
		\lim_{\delta\to0^+} B_\delta = 0 \qquad \text{and} \qquad C_\delta \le C \mathcal{L}^n\left(\Omega\setminus D\right).
	\end{equation}
In fact, for all $\delta \in (0,1)$, we have
	$$
	\begin{aligned}
		 B_{\delta}&=  \int_{D} \delta^{p(x)}\,dx \int_{\Omega \setminus B(x,t)} \frac{ \varphi \left (x, \left | u(x) - u(y) \right | \slash \delta \right )}{\left | x-y \right | ^{n+p(x)}} \, dy \\
		& \le \int_{D} \delta^{p(x)}\,dx \int_{\Omega \setminus B(x,t)} \frac{b}{t^{n+p(x)}} \, dy \
		 \le \frac{\delta}{\min\left\{t^{n+1},t^{n+p^+}\right\}} b \left(\mathcal{L}^n\left(\Omega\right)\right)^2.
	\end{aligned}
	$$
	Secondly, since
	$$
	\begin{aligned}
		C_{\delta} & = \int_{\Omega \setminus D} \delta^{p(x)}\,dx \int_{\Omega} \frac{ \varphi \left ( x,\left | u(x) - u(y) \right | \slash \delta \right )}{\left | x-y \right | ^{n+p(x)}} \, dy \, dx \\
		& \le \int_{\Omega \setminus D} \delta^{p(x)}\,dx \int_{\Omega} \frac{ \tilde{\varphi} \left ( x,\left | u(x) - u(y) \right | \slash \delta \right )}{\left | x-y \right | ^{n+p(x)}} \, dy \, dx \\
		& \le \int_{\Omega \setminus D} \delta^{p(x)}\,dx \int_{\Omega} \frac{\tilde{\varphi} \left (x, L \left | x-y \right | / \delta \right )}{\left | x-y \right |^{n+p(x)}} \, dy, \qquad \forall \delta \in (0,1),
	\end{aligned}
	$$
	where $L$ is the Lipschitz constant of $u$ over $\Omega$, 
	by making the change of variables $z=L\left(y-x\right) \slash \delta$ and using the definition of $\tilde{\varphi}$, we have
	$$
	\begin{aligned}
		C_{\delta} \le \int_{\Omega \setminus D} \delta^{p(x)} \, dx & \int_{\Omega} \left(\frac{L}{\delta}\right)^{p(x)} \frac{\tilde{\varphi} \left (x, \left| z \right | \right )}{\left | z \right |^{n+p(x)}} \, dz \\
		& \le \int_{\Omega \setminus D} \delta^{p(x)}\,dx \int_{\Omega\cap\left\{|z| \le 1\right\}} \left(\frac{L}{\delta}\right)^{p(x)}  \frac{a}{\left | z \right |^{n-1}} \, dz \\
		& + \int_{\Omega \setminus D} \delta^{p(x)}\,dx \int_{\Omega\cap\left\{|z| > 1\right\}} \left(\frac{L}{\delta}\right)^{p(x)} \frac{b}{\left | z \right |^{n+p(x)}} \, dz \\
		& \le \int_{\Omega \setminus D} L^{p(x)} \,dx \left(a \int_{\left\{|z| \le 1\right\}}  \frac{1}{\left | z \right |^{n-1}} \, dz + b \int_{\left\{|z| > 1\right\}} \frac{1}{\left | z \right |^{n+1}} \, dz \right) \\
		& \le C \mathcal{L}^n\left(\Omega\setminus D\right),
	\end{aligned}
	$$
	where $C$ depends on $L$, $a$, $b$ and $n$. Being
	$$
		\left | \Lambda_{\delta}(u) - \int_{\Omega} \left | \nabla u(x) \right |^{p(x)} \,dx \right | \le  \left | 	A_{\delta}(u) - \int_{D} \left | \nabla u(x) \right |^{p(x)}\,dx \right | + B_{\delta}+ C_{\delta} +  \int_{\Omega \setminus D} \left | \nabla u(x) \right |^{p(x)}\,dx ,
	$$
	by using \eqref{stima_A} and \eqref{stime_delta}, we get
	$$
		\limsup_{\delta \to 0^+} \left | \Lambda_{\delta}(u) - \int_{\Omega} \left | \nabla u(x) \right |^{p(x)}\,dx \right | \le C \mathcal{L}^n\left(\Omega\setminus D\right) + \int_{\Omega \setminus D} \left | \nabla u(x) \right |^{p(x)}\,dx.
	$$
	By the arbitrariness of $D\Subset\Omega$ the assertion follows.
\end{proof}

\end{document}